\documentclass[12pt, reqno]{amsart}

\newtheorem{theorem}{Theorem}[section]

\newtheorem{corollary}[theorem]{Corollary}   
\newtheorem{definition}[theorem]{Definition}
\newtheorem{example}[theorem]{Example}

\newtheorem{remark}[theorem]{Remark}

\numberwithin{equation}{section}

\usepackage{times}
\usepackage{enumerate}
\usepackage{mathrsfs}
\usepackage{tfrupee}
\usepackage{tikz}
\usetikzlibrary{chains,fit}
\usetikzlibrary{shapes,snakes}
\usetikzlibrary{graphs}
\usepackage{graphicx,adjustbox}
\usepackage{hyperref}
\usepackage{amsmath,amssymb}
\usepackage{amscd}
\usepackage{graphicx}
\usepackage[all]{xy}

\usepackage{float}
\usepackage{caption}
\usepackage{subcaption}

\usepackage{cleveref}

\usepackage{geometry}
 \usepackage[backend=bibtex]{biblatex}
\begin{document}

\title{On a Regularity-Conjecture of Generalized Binomial Edge Ideals}
\author{
Anuvinda J, Ranjana Mehta, \and Kamalesh Saha 
}
\date{}

\address{\small \rm Department of Mathematics,
SRM University - AP, Amaravati, Andhra Pradesh 522502, India.}
\email{anuvinda\_j@srmap.edu.in} 

\address{\small \rm Department of Mathematics,
SRM University - AP, Amaravati, Andhra Pradesh 522502, India.}
\email{ranjana.m@srmap.edu.in}

\address{\small \rm Chennai Mathematical Institute, Siruseri, Chennai, Tamil Nadu 603103, India}
\email{ksaha@cmi.ac.in}
\thanks{Kamalesh Saha is funded by NBHM Post-Doctoral Fellowship, sponsored by the National Board of Higher Mathematics, Government of India.}

\date{}



\allowdisplaybreaks

\begin{abstract}
In this paper, we prove the upper bound conjecture proposed by Saeedi Madani \& Kiani on the Castelnuovo-Mumford regularity of generalized binomial edge ideals. We give a combinatorial upper bound of regularity for generalized binomial edge ideals, which is better than the bound claimed in that conjecture. Also, we show that the bound is tight by providing an infinite class of graphs.
\end{abstract}

\maketitle

\section{Introduction}
The ideals generated by all $t$-minors of a generic matrix $X$, denoted by $I_{t}(X)$, known as determinantal ideals, have been studied extensively in the literature due to their rich connections with invariant theory, representation theory, and combinatorics (see the survey paper \cite{bc03}). In the current scenario, people started considering the ideals generated by an arbitrary set of minors of a generic matrix. In this direction, one of the significant works is on the binomial edge ideals of graphs, which was introduced in 2010 via \cite{hhhrkara} and \cite{ohtani} independently. Binomial edge ideals can be seen as a generalization of the determinantal ideals $I_{2}(X)$ of a $2\times n$ generic matrix $X$. Later, Johannes Rauh in \cite{rauh13} introduced the notion of generalized binomial edge ideals as a generalization of binomial edge ideals. One of the primary motivations to study both binomial edge ideals and generalized binomial edge ideals is their connection to algebraic statistics, particularly their appearance in the study of conditional independence ideals. As a generalization of generalized binomial edge ideals, Ene et al. in \cite{ehhq14} introduced a more general class of binomial ideals, known as ``binomial edge ideals of pair of graphs". These ideals can be interpreted as ideals generated by certain sets of $2$-minors of a finite matrix of indeterminates and include the ideals studied in \cite{des98}, \cite{hs00}, \cite{hs04}. Lots of research have been done on binomial edge ideals of graphs in the last decade (see the survey paper \cite{binomsurvey}), but generalized binomial edge ideals and binomial edge ideals of pair of graphs are not that much explored, and there are numerous open problems in this direction. 
\par

Let $G_1$ and $G_2$ be two simple graphs on the vertex set $[m]=\{1,\ldots,m\}$ and $[n]=\{1,\ldots,n\}$ respectively. Let $X=(x_{ij})$ be an $(m\times n)$-matrix of indeterminates, and $K[X]$ be the polynomial ring in the variables $x_{ij}$, $1\leq i\leq m$ and $1\leq j\leq n$, where $K$ is a field. Let $e_1 = \{i, j\}\in E(G_1)$ with $i<j$ and $e_2=\{k, l\}\in E(G_2)$ with $k<l$. Corresponding to the pair $(e_1, e_2)$, consider the following $2$-minor of $X$:
$$ p_{e_1,e_2} = [i, j\vert k, l] = x_{ik}x_{jl}- x_{il}x_{jk}.$$
Then the \textit{binomial edge ideal of the pair} $(G_1, G_2)$, denoted by $\mathcal{J}_{G_1,G_2}$, is defined as
$$\mathcal{J}_{G_1,G_2} = ( p_{e_1,e_2} \mid e_1\in E(G_1), e_2\in E(G_2)).$$

\noindent Note that the \textit{generalized binomial edge ideal} of a graph $G$ introduced in \cite{rauh13} is nothing but the ideal $\mathcal{J}_{K_{m},G}$, where $K_{m}$ is the complete graph on the vertex set $[m]$. If $m=2$, then the ideal $\mathcal{J}_{K_{2},G}$ is the well-known \textit{binomial edge ideal} $J_{G}$ of $G$. Generalized binomial edge ideals behave well than arbitrary binomial edge ideals of pair of graphs as $\mathcal{J}_{G_1,G_2}$ is radical if and only if one of $G_1$ and $G_2$ is complete. Also, for unmixedness of $\mathcal{J}_{G_1,G_2}$, it should be a generalized binomial edge ideal (see \cite[Proposition 4.1]{skpair13}). The importance of considering $\mathcal{J}_{K_{m},G}$ is their appearance in algebraic statistic (see \cite{rauh13}, \cite{rauhay14}). We recommend the reader to see \cite{acr21}, \cite{bei18}, \cite{ci20}, \cite{ehhq14}, \cite{arvindgen20}, \cite{rauh13}, \cite{skpair13} \cite{szhu23} to get updated on the work on generalized binomial edge ideals so far.
\par

One of the most important homological invariants that has attracted much attention in the study of ideals associated with graphs is the \textit{Castelnuovo-Mumford regularity} (in short, regularity). For a graded ideal $I$ in a standard graded ring $R$, the regularity of $R/I$, denoted by $\mathrm{reg}(R/I)$, is defined as follows:
$$\mathrm{reg}(R/I) = \mathrm{max}\{j-i\mid  \beta_{i,j}(R/I)\neq 0\}.$$

In this paper, we prove the Saeedi Madani-Kiani conjecture \cite{skpair13} on the regularity of generalized binomial edge ideals, which Kumar proved in \cite{arvindgen20} for chordal graphs. However, we prove this conjecture for general graphs, and our given bound is even better than that given by Kumar in \cite{arvindgen20}. Moreover, our results generalize those given in \cite{rsmkconj21} for the binomial edge ideals of graphs. For a graph $G$, Saeedi et al. in \cite{skpair13} conjectured that $\mathrm{reg}(S/J_{G})\leq c(G)$, where $c(G)$ denotes the number of maximal cliques of $G$. Later, in \cite{rsmkconj21}, the conjecture was proved by providing a better upper bound $\eta(G)$ (see Definition \ref{defeta}). Again in \cite{skpair13}, the authors extend the conjecture for generalized binomial edge ideals: $\mathrm{reg}(S/\mathcal{J}_{K_m,G})\leq \mathrm{min}\{\binom{m}{2}c(G), e(G)\}$, where $e(G)$ denotes the number of edges of $G$. In \cite{arvindgen20}, Kumar proved this conjecture for chordal graphs by showing $\mathrm{reg}(S/\mathcal{J}_{K_{m},G})\leq \mathrm{min}\{(m-1)c(G), n-1\}$ for any chordal graph $G$ and this bound is better than the bound conjectured by Saeedi Madani-Kiani. In this paper, we solve this conjecture for the general case.
\par

The paper is written in the following manner: In Section \ref{preli}, we recall some notions, definitions, and results, which are necessary to explain our work. In Section \ref{secreg}, we introduce the concept of $r$-compatible map from $\mathcal{K\times G}$ to $\mathbb{N}_0=\mathbb{N}\cup \{0\}$, where $\mathcal{K}$ denotes the class of all complete graphs, and $\mathcal{G}$ denotes the class of all simple graphs. This map extends the notion of the compatible map defined in \cite{rsmkconj21}. In Theorem \ref{thmcomp}, we show that $\mathrm{reg}(S/\mathcal{J}_{K_m,G})\leq \Phi(K_m,G)$ for any $r$-compatible map $\Phi$. Next, we show the existence of a $r$-compatible map using a graph invariant $\eta(G)$, which is the maximum size of a clique disjoint edge set (see Definition \ref{defeta}). This gives a combinatorial upper bound of $\mathrm{reg}(S/\mathcal{J}_{k_m,G})$. In particular, we show that for any graph $G$, $\mathrm{reg}(S/\mathcal{J}_{K_{m},G})\leq \mathrm{min}\{(m-1)\eta(G), n-1\}$, which is better than both the bounds, given by Kumar for chordal graphs in \cite{arvindgen20} and claimed by Saaedi Madani-Kiani for general graphs in \cite{skpair13}. Moreover, our bound is the generalization of the upper bound given for binomial edge ideals of graphs in \cite{rsmkconj21}. In Section \ref{secexample}, we provide a non-trivial class of graphs for which the upper bound is tight.

\section{Preliminaries}\label{preli}
Let $G$ be a simple graph and $T\subseteq V(G)$. We write $G - T$ to denote the induced subgraph of $G$ on the vertex set $V(G)\setminus T$. For a vertex $v\in V(G)$, we use the notation $G-v$ to denote the induced subgraph $G-\{v\}$. A vertex $v\in V(G)$ is said to be a \textit{cut vertex} of $G$ if the number of connected components of $G-v$ is strictly greater than the number connected components of $G$. For $T\subseteq V(G)$, let $c_{G}(T)$ denote the number of connected components of $G-T$ and $G_{1},G_{2},\cdots,G_{c_{G}(T)}$ be the connected components of $G-T$. We denote by $\tilde{G}$ the complete graph on $V(G)$. For a positive integer $m\geq 2$ and $T\subseteq V(G)=[n]$, let us consider the following ideal:
$$P_{T}(K_{m},G)=( \{x_{ij}:(i,j)\in [m] \times T\},\mathcal{J}_{K_{m},\tilde{G_{1}}},\mathcal{J}_{K_{m},\tilde{G_{2}}},\cdots,\mathcal{J}_{K_{m},\tilde{G}_{c_{G}(T)}})$$
in the polynomial ring $S=K[x_{ij}:i\in [m],j \in [n]]$. Then $P_{T}(K_{m},G)$ is a prime ideal containing $\mathcal{J}_{K_{m},G}$ by \cite[Lemma 5]{rauh13}. A subset $T\subseteq [n]$ is said to be a \textit{cutset} of $G$ if each $i\in T$ is a cut vertex of the induced subgraph $G-(T\setminus\{i\})$. The set of all cutsets of $G$ is denoted by $\mathcal{C}(G)$. By \cite[Theorem 7]{rauh13}, $\mathcal{J}_{K_{m},G}$ is radical and the set of minimal prime ideals of $\mathcal{J}_{K_{m},G}$ is $\{P_{T}(K_{m},G)\mid T\in \mathcal{C}(G)\}$. Thus,
$$ \mathcal{J}_{K_{m},G} = \bigcap_{T \in \mathcal{C}(G)}P_{T}(K_{m},G).$$ 

Let $G$ be a graph. For a vertex $v$ of $G$, the set $\mathcal{N}_G(v):=\{u\mid \{u,v\}\in E(G)\}$ is called the \textit{neighbour set} of $v$ in $G$. The graph $G_{v}$, defined as
$$V(G_{v}):=V(G)\,\,\text{and}\,\, E(G_v):=\{\{u_{1},u_{2}\}\mid u_1,u_2\in \mathcal{N}_G(v)\}\cup E(G),$$
plays an important role in the study of binomial and generalized binomial edge ideals due to the following result.

\begin{theorem}[{\cite[Theorem 3.2]{arvindgen20}}] \label{thrm:2} Let $G$ be a finite simple graph and $v$ be an internal vertex of $G$. Then,
$$\mathcal{J}_{K_{m},G}=\mathcal{J}_{K_{m},G_{v}} \cap ((x_{iv}: i\in [m])+\mathcal{J}_{K_{m},G-v}).$$ 
\end{theorem}
 
A vertex $v\in V(G)$ is said to be a \textit{free} vertex of $G$ if the induced subgraph of $G$ on the vertex set $\mathcal{N}_{G}(v)$ is a complete graph. Note that $v$ is always a free vertex of $G_v$. If $v$ is not a free vertex of $G$, we call it a \textit{non-free} vertex or an \textit{internal} vertex of $G$. We will denote the number of internal vertices of $G$ by $\mathrm{iv}(G)$. A \textit{clique} of a graph $G$ is an induced subgraph of $G$ that is a complete graph.


%


\section{Regularity bound for generalized binomial edge ideals}\label{secreg}

In this section, we generalize the notion of compatible map given in \cite{rsmkconj21}. Using that map, we give a combinatorial upper bound of $\mathrm{reg}(S/\mathcal{J}_{K_m,G})$, which solve the conjecture on regularity of generalized binomial edge ideals given in \cite{skpair13}. Moreover, our given bound is better than the bound claimed in that conjecture.

\begin{definition}\label{defcomp}{\rm
Let $\mathcal{G}$ be the set of all simple graphs and $\mathcal{K}$ be the set of all complete graphs. Let $ G \in \mathcal{G}$ be a graph on $[n]$ and $K_{m}\in \mathcal{K}$ be a complete graph on $m$ vertices. We set $\widehat{G}= G - \mathrm{Is}(G)$, where $\mathrm{Is}(G)$ is the set of isolated vertices of $G$. We define a map $\Phi:\mathcal{K \times G}\rightarrow\mathbb{N}_{0}$ as $r$-\textit{compatible} if
\begin{enumerate}
\item $\Phi(K_{m},\widehat{G})\leq \Phi(K_{m},G)$ for every $ G \in \mathcal{G}$
\item If $G= \bigsqcup_{i=1}^{t}K_{n_{i}}$, where $n_{i}\geq 2$ for all $1\leq i\leq t$, then $\Phi(K_{m},G)\geq k$, where $k=\sum_{i=1}^{t} \mathrm{min}\{ m-1,n_{i}-1\}$.
\item If $G\neq \bigsqcup_{i=1}^{t}K_{n_{i}}$, then there exists a non-free vertex $v \in V(G)$ such that the following conditions hold:
\begin{enumerate}
\item $\Phi(K_{m},G-v)\leq \Phi(K_{m},G)$
\item $\Phi(K_{m},G_{v})\leq \Phi(K_{m},G)$
\item $\Phi(K_{m},G_{v}-v)+1 \leq \Phi(K_{m},G)$.
\end{enumerate}
\end{enumerate}
}
\end{definition}

The following theorem establishes a general upper bound for the regularity of generalized binomial edge ideals.

\begin{theorem}\label{thmregcomp}
 Let $G$ be a graph on [n] and $\Phi$ be a $r$-compatible map. Then $$reg (S/\mathcal{J}_{K_{m},G})\leq \Phi(K_{m},G)$$ 
 \end{theorem}
\begin{proof} Let $\mathrm{iv}(G)$ denote the number of non-free vertices of a graph $G$. We prove this theorem by mathematical induction on $\mathrm{iv}(G)$. If $\mathrm{iv}(G)=0$, then $G$ is a disjoint union of complete graphs. Let $\widehat{G}= \bigsqcup_{i=1}^{t}K_{n_{i}}$. Then $n_{i}\geq 2$ for all $1\leq i\leq t$. Since $\mathcal{J}_{K_{m},G}=\mathcal{J}_{K_{m},\widehat{G}}$ in $S$, we have $\mathrm{reg}(S/\mathcal{J}_{K_{m},G})=\mathrm{reg}(S/\mathcal{J}_{K_{m},\widehat{G}})$. If a graph $H$ is disconnected and has the connected components $H_{1},\cdots,H_{c}$ then $\mathcal{J}_{K_{m},H_1},\ldots, \mathcal{J}_{K_{m},H_c}$ being generated in pairwise disjoint set of variables, we have
$$ S/\mathcal{J}_{K_{m},H} \cong S_{1}/\mathcal{J}_{K_{m},H_{1}}\otimes \cdots \otimes S_{c}/\mathcal{J}_{K_{m},H_{c}},$$ where $S_{i}=K[x_{1j},\cdots,x_{mj}:j\in V(H_{i})]$. Using \cite[Proposition 3.3]{arvindgen20}, we get 
$$\mathrm{reg}(S/\mathcal{J}_{K_{m},G})= \sum_{i=1}^{t} \mathrm{reg}(S_{i}/\mathcal{J}_{K_{m},K_{n_i}})=k,$$ 
where $k=\sum_{i=1}^{t} min\{m-1,n_{i}-1\}$. Now conditions (1) and (2) in the definition of $r$-compatible map implies $k\leq \Phi(K_{m},\widehat{G})\leq \Phi(K_{m},G)$. Therefore, $\mathrm{reg} (S/\mathcal{J}_{K_{m},G})\leq \Phi(K_m,G)$ holds in this case.\par

Now, we consider $\mathrm{iv}(G)\neq 0$. Then there exists a non-free vertex $v \in [n]$ for $\Phi$ satisfying the condition (3) in Definition \ref{defcomp}. We assume that the assertion holds for all graphs $G'$ with $\mathrm{iv}(G')<\mathrm{iv}(G)$. By Theorem \ref{thrm:2}, $$\mathcal{J}_{K_{m},G}=\mathcal{J}_{K_{m},G_{v}} \cap ((x_{iv}: i\in [m])+\mathcal{J}_{K_{m},G-v}),$$
and it is easy to verify that 
$$\mathcal{J}_{K_{m},G_{v}}+((x_{iv}: i\in [m])+\mathcal{J}_{K_{m},G-v})= (x_{iv}: i\in [m])+\mathcal{J}_{K_{m},G_{v}-v}.$$
\noindent Thus, we obtain the following short exact sequence:
$$0 \longrightarrow S/\mathcal{J}_{K_{m},G} \longrightarrow S/\mathcal{J}_{K_{m},G_{v}} \oplus S/((x_{iv}: i\in [m])+\mathcal{J}_{K_{m},G-v})$$ $$  \longrightarrow S/((x_{iv}: i\in [m])+\mathcal{J}_{K_{m},G_{v}-v}) \longrightarrow 0,$$ 
namely,$$0 \longrightarrow S/\mathcal{J}_{K_{m},G} \longrightarrow S/\mathcal{J}_{K_{m},G_{v}} \oplus S_{v}/\mathcal{J}_{K_{m},G-v} \longrightarrow S_{v}/\mathcal{J}_{K_{m},G_{v}-v} \longrightarrow 0 $$ where $S_{v}= K[x_{ij}:i\in [m],j\in V(G-v)]$. Since $\mathrm{reg}(S/\mathcal{J}_{K_{m},G_{v}} \oplus S_{v}/\mathcal{J}_{K_{m},G-v})=\mathrm{max}\{\mathrm{reg}(S/\mathcal{J}_{K_{m},G_{v}}), \mathrm{reg}(S_{v}/\mathcal{J}_{K_{m},G-v})\}$, by the well-known regularity lemma, we get
{\small
\begin{equation}\label{eq1}
\mathrm{reg}(S/\mathcal{J}_{K_{m},G})\leq \mathrm{max} \{ \mathrm{reg}(S/\mathcal{J}_{K_{m},G_{v}}),\mathrm{reg}(S_{v}/\mathcal{J}_{K_{m},G-v}), \mathrm{reg}(S_{v}/\mathcal{J}_{K_{m},G_{v}-v})+1\}
\end{equation}
}
\noindent Due to \cite[Lemma 3.4]{arvindgen20}, by the induction hypothesis and Definition \ref{defcomp}, we get
\begin{equation}
\mathrm{reg}(S/\mathcal{J}_{K_{m},G_{v}}) \leq \Phi(K_{m},G_{v})\leq \Phi(K_{m},G),\label{eq2}
\end{equation}
\begin{equation}\label{eq3}
\mathrm{reg}(S_{v}/\mathcal{J}_{K_{m},G-v}) \leq \Phi(K_{m},G-v)\leq \Phi(K_{m},G),
\end{equation}
\begin{equation}\label{eq4}
\mathrm{reg}(S_{v}/\mathcal{J}_{K_{m},G_{v}-v}) +1 \leq \Phi(K_{m},G_{v}-v)+1 \leq \Phi(K_{m},G).
\end{equation}
From \eqref{eq1}, \eqref{eq2}, \eqref{eq3} and \eqref{eq4}, we have $\mathrm{reg}(S/\mathcal{J}_{K_{m},G})\leq \Phi(K_{m},G)$.
\end{proof}

\begin{definition}[{\cite[Definition 2.4]{rsmkconj21}}]\label{defeta} {\rm
Let $G \in \mathcal{G}$ be a graph and $\mathcal{M}\subseteq E(G)$ be a collection of edges with the property that no two elements of $\mathcal{M}$ belong to the same clique of $G$. Then, we call the set $\mathcal{M}$, a \textit{clique disjoint edge set} in $G$. Moreover, we assign the following graph invariant associated with clique disjoint sets of a graph $G$:
$$\eta(G):= \mathrm{max} \{ \vert\mathcal{N}\vert \mid \mathcal{N}\,\, \text{is a clique disjoint edge set in G}\}.$$
}
\end{definition}

In the following theorem, we show the existence of $r$-compatible map and give a combinatorial upper bound of $\mathrm{reg}(S/\mathcal{J}_{K_{m},G})$ using $\eta(G)$.

\begin{theorem}\label{thmcomp}
The map $\Phi:\mathcal{K \times G}\rightarrow\mathbb{N}_{0}$ defined by $$\Phi(K_{m},G) = \mathrm{min} \{(m-1)\eta(G), n-1\}$$ is $r$-compatible. Hence, $\mathrm{reg}(S/\mathcal{J}_{K_{m},G})\leq \mathrm{min}\{(m-1)\eta(G), n-1\}$.
\end{theorem} 

\begin{proof} 
Fix $K_{m}$ and let $G$ be any graph with vertex set $[n]$. From the definition of $\eta$, we have, $\eta(\widehat{G})= \eta(G)$. Hence, $(m-1)\eta(\widehat{G})=(m-1)\eta(G)$. Relabeling the vertices of $G$, we can get $V(\widehat{G})=[\widehat{n}]$, where $\widehat{n}\leq n$. Thus, 
\begin{align*}
\Phi(K_{m},\widehat{G})& = \mathrm{min} \{(m-1)\eta(\widehat{G}), \widehat{n}-1\} \\& \leq \mathrm{min} \{(m-1)\eta(G), n-1\}\\& = \Phi(K_{m},G).
\end{align*} 
Now, let $G= \bigsqcup_{i=1}^{t}K_{n_{i}}$, where $n_{i}\geq 2$ for all $1\leq i\leq t$. Then it is clear that $\eta(G)= t$. Let $k= \sum_{i=1}^{t} min\{m-1,n_{i}-1\}$. Then we have, 
\begin{align*}
n_{1}-1 &\geq  min\{m-1,n_{1}-1\}\\ n_{2}-1 &\geq min\{m-1,n_{2}-1\} \\ &\vdots \\ n_{t}-1 &\geq  min\{m-1,n_{t}-1\} 
\end{align*}
This implies 
$$n_{1}+n_{2}+\cdots+n_{t}-t  \geq \sum_{i=1}^{t} \mathrm{min}\{m-1,n_{i}-1\},$$
i.e., $n-t\geq k$ and hence, $ n-1\geq k$. Now, $m-1 \geq \mathrm{min}\{m-1,n_{i}-1\}$ for all $1\leq i\leq t$ gives $$(m-1)t \geq \sum_{i=1}^{t} \mathrm{min}\{m-1,n_{i}-1\}=k.$$
Since $\eta(G)=t$, we have 
\begin{align*}
\Phi(K_{m},G) &= \mathrm{min}\{(m-1)\eta(G),n-1\}\\ &\geq k.
\end{align*}
To show $\Phi$ satisfies the condition $(3)$ of Definition \ref{defcomp}, we assume that $G$ is not a disjoint union of complete graphs. Then there exists a non-free vertex $v$ of $G$. From the proof of \cite[Theorem 2.5]{rsmkconj21}, we get $\eta(G-v)\leq \eta(G)$ and hence, $(m-1)\eta(G-v)\leq (m-1)\eta(G)$. This yields, 
\begin{align*}
\Phi(K_{m},G-v)& = min \{(m-1)\eta(G-v), n-2\} \\ &\leq min \{(m-1)\eta(G), n-1\} \\ & =\Phi(K_{m},G). 
\end{align*} 
Next for $ G_{v}$, again from the proof of \cite[Theorem 2.5]{rsmkconj21}, we get $\eta(G_{v})< \eta(G)$. Thus,
\begin{align*}
\Phi(K_{m},G_{v}) & = min \{(m-1)\eta(G_{v}), n-1\} \\ & \leq min \{(m-1)\eta(G), n-1\} \\ & = \Phi(K_{m},G).
\end{align*}
Now for $G_{v}-v$, we observe that $\eta(G_{v}-v)\leq \eta(G_{v})< \eta(G)$, which gives $(m-1)\eta(G_{v}-v)+1\leq (m-1)\eta(G)$. Therefore, we have 
\begin{align*}
\Phi(K_{m},G_{v}-v)+1 &= \mathrm{min} \{(m-1)\eta(G_{v}-v), n-2\}+1\\
&=\mathrm{min}\{(m-1)\eta(G_{v}-v)+1, n-2+1\}\\
&\leq \mathrm{min}\{(m-1)\eta(G), n-1\}\\
&=\Phi(K_{m},G).
\end{align*}
The given map $\Phi$ satisfies all the conditions of Definition \ref{defcomp}, and thus, $\Phi$ is an $r$-compatible map. Hence, by Theorem \ref{thmregcomp}, it follows that $\mathrm{reg}(S/\mathcal{J}_{K_m,G})\leq \mathrm{min}\{(m-1)\eta(G),n-1\}$, where $n$ is the number of vertices of $G$.
\end{proof}

\begin{remark}{\rm
Since $\eta(G)\leq c(G)$ for any graph $G$, Theorem \ref{thmcomp} solves the Saeedi Madani-Kiani conjecture on $\mathrm{reg}(S/\mathcal{J}_{K_{m},G})$ by providing a better upper bound.
}
\end{remark}

\begin{corollary} 
Let $G$ be a graph on [n]. Let $\Phi:\mathcal{K\times G}\rightarrow \mathbb{N}_{0}$ be the $r$-compatible map defined in Theorem \ref{thmcomp}. If $\Phi(K_{m},G_{v})= \Phi(K_{m},G)$ for a non-free vertex $v\in V(G)$, then $\mathrm{reg}(S/\mathcal{J}_{K_{m},G}) = n-1$.
\end{corollary}
\begin{proof}
From the proof of Theorem \ref{thmcomp}, it is clear that any non-free vertex $v\in V(G)$ satisfies the condition $(3)$ of $r$-compatible map if $\Phi$ is defined by $\Phi(K_m,G)=\mathrm{min}\{(m-1)\eta(G), n-1\}$. Since $\eta(G_{v})<\eta(G)$ holds for any non-free vertex $v\in V(G)$ by  \cite[Theorem 2.5]{rsmkconj21}, the result follows.
\end{proof}

\section{A class of graphs attaining the upper bound}\label{secexample}

In this section, we give an infinite class of graphs (not only the disjoint union of complete graphs) for which the bound given in Theorem \ref{thmcomp} is attained.
\medskip

In Graph Theory, there is a popular class of chordal graphs, namely, \textit{block graphs} or \textit{clique trees}. A graph $G$ is said to be a \textit{block graph} if every block (i.e., maximal biconnected induced subgraph) of $G$ is a clique. In the study of binomial edge ideals, block graphs are well studied. In \cite[Theorem 1.1]{ehh_cmbin}, Ene et al. completely classified the class of block graphs with Cohen-Macaulay binomial edge ideals. Let $\mathcal{B}$ be the class of block graphs such that a vertex can belong to at most two maximal cliques. Then $\mathcal{B}$ is precisely the class of block graphs whose binomial edge ideals are Cohen-Macaulay. Due to \cite[Corollary 3.3]{jnr19}, we get $\mathrm{reg}(S/J_{G})=c(G)$ for all $G\in \mathcal{B}$. Corresponding to every natural number, we will consider a subclass of block graphs with Cohen-Macaulay binomial edge ideals. For $m\in \mathbb{N}$, let $\mathcal{B}_{m}\subseteq \mathcal{B}$ be the collection of those graphs $G$ whose every maximal clique contains at least $m+r$ number of vertices, where $r$ is the number of cut vertices of $G$ belong to that clique. 

\begin{theorem}\label{thmBm}
Let $m\in\mathbb{N}$. Then for every $G\in \mathcal{B}_{m}$, we have
$$\mathrm{reg}(S/\mathcal{J}_{K_m,G})=(m-1)c(G)=(m-1)\eta(G).$$
\end{theorem}

\begin{proof}
Let $G\in \mathcal{B}_{m}$ be a graph and $V$ be the set of cut vertices of $G$. Since $G$ is a block graph, it is clear from the definition that $c(G)=\eta(G)=k$ (say). Consider the graph $H=G - V$. Then $H$ is a disjoint union of $k$ complete graphs. Let $H=\bigsqcup_{i=1}^k K_{n_{i}}$. Since $G\in \mathcal{B}_{m}$, each maximal cliques of $G$ contains at least $m$ free vertices of $G$. Hence, after removing all the cut vertices of $G$, every connected component will contain at least $m$ vertices. Therefore, we have $n_{i}\geq m$ for all $1\leq i\leq k$. Then by \cite[Proposition 3.3]{arvindgen20},
\begin{align*}
\mathrm{reg}(S/\mathcal{J}_{K_{m}, H})&=\sum_{i=1}^{k}\mathrm{min}\{ m-1,n_i-1\}\\
&=\sum_{i=1}^{k} (m-1)\\
&=(m-1)k.
\end{align*}
$H$ being an induced subgraph of $G$, by \cite[Proposition 8]{skpair13}, it follows that $\mathrm{reg}(S/\mathcal{J}_{K_{m},G})\geq \mathrm{reg}(S/\mathcal{J}_{K_{m}, H})=(m-1)k$. Again by Theorem \ref{thmcomp}, we have $\mathrm{reg}(S/\mathcal{J}_{K_m,G})\leq (m-1)k$. Hence, the assertion follows.
\end{proof}

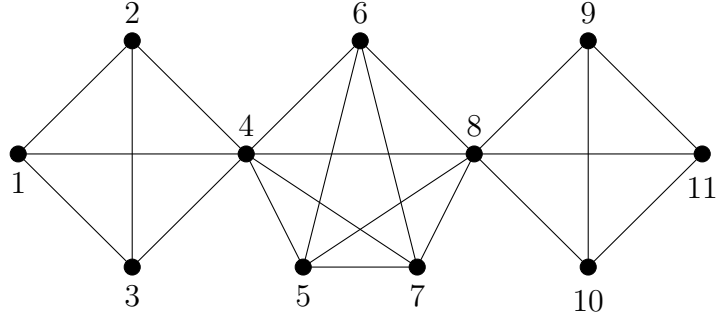
\begin{figure}[H]
	\centering
\begin{tikzpicture}
[scale=1.5,auto=left,every node/.style={circle,scale=1}]

\filldraw[black] (0,0) circle (2pt)node[anchor=north]{$1$};
\filldraw[black] (1,1) circle (2pt)node[anchor=south]{$2$};
\filldraw[black] (1,-1) circle (2pt)node[anchor=north]{$3$};
\filldraw[black] (2,0) circle (2pt)node[anchor=south]{$4$};
\filldraw[black] (2.5,-1) circle (2pt)node[anchor=north]{$5$};
\filldraw[black] (3,1) circle (2pt)node[anchor=south]{$6$};
\filldraw[black] (3.5,-1) circle (2pt)node[anchor=north]{$7$};
\filldraw[black] (4,0) circle (2pt)node[anchor=south]{$8$};
\filldraw[black] (5,1) circle (2pt)node[anchor=south]{$9$};
\filldraw[black] (5,-1) circle (2pt)node[anchor=north]{$10$};
\filldraw[black] (6,0) circle (2pt)node[anchor=north]{$11$};

\draw[black] (0,0) -- (1,1) -- (2,0) -- (1,-1) -- cycle;
\draw[black] (2,0) -- (2.5,-1) -- (3.5,-1) -- (4,0) -- (3,1) -- cycle;
\draw[black] (4,0) -- (5,1) -- (6,0) -- (5,-1) -- cycle;
\draw[black] (0,0) -- (2,0);
\draw[black] (1,1) -- (1,-1);
\draw[black] (2,0) -- (3.5,-1);
\draw[black] (2,0) -- (4,0);
\draw[black] (3,1) -- (2.5,-1);
\draw[black] (3,1) -- (3.5,-1);
\draw[black] (4,0) -- (2.5,-1);
\draw[black] (4,0) -- (6,0);
\draw[black] (5,1) -- (5,-1);

\end{tikzpicture}
\caption{Graph $G$ such that $\mathrm{reg}(S/\mathcal{J}_{K_3,G})=(3-1)\eta(G)=6$.}\label{fig1}
\end{figure}

\begin{example}{\rm
Let us consider the pair of graphs $(K_3,G)$, where $G$ is the graph given in Figure \ref{fig1}. Note that $G$ is a block graph with $c(G)=\eta(G)=3$. Number of cut vertices of $G$ belong to each $K_4$ block is $1$ and the block $K_5$ contains $2$ cut vertices of $G$. Then $G\in \mathcal{B}_{3}$. Therefore, by Theorem \ref{thmBm}, we get $\mathrm{reg}(S/\mathcal{J}_{K_{3},G})=(3-1)\eta(G)=6$. In this case, $\mathrm{min}\{(m-1)\eta(G),n-1\}=\mathrm{min}\{6,10\}=6$, i.e., the pair $(K_{3},G)$ attains the upper bound given in Theorem \ref{thmcomp}.
}
\end{example}

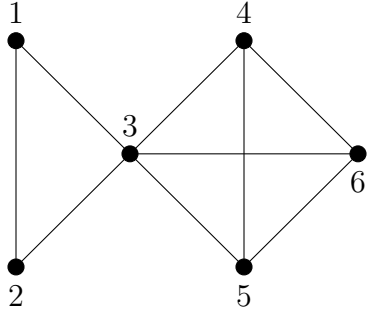
\begin{figure}[H]
	\centering
\begin{tikzpicture}
[scale=1.5,auto=left,every node/.style={circle,scale=1}]

\filldraw[black] (1,1) circle (2pt)node[anchor=south]{$1$};
\filldraw[black] (1,-1) circle (2pt)node[anchor=north]{$2$};
\filldraw[black] (2,0) circle (2pt)node[anchor=south]{$3$};
\filldraw[black] (3,1) circle (2pt)node[anchor=south]{$4$};
\filldraw[black] (3,-1) circle (2pt)node[anchor=north]{$5$};
\filldraw[black] (4,0) circle (2pt)node[anchor=north]{$6$};

\draw[black] (1,1) -- (1,-1) -- (2,0) -- cycle;
\draw[black] (2,0) -- (3,1) -- (4,0) -- (3,-1) -- cycle;
\draw[black] (2,0) -- (4,0);
\draw[black] (3,-1) -- (3,1);

\end{tikzpicture}

\caption{Graph $G\in\mathcal{B}\setminus \mathcal{B}_{3}$ and $\mathrm{reg}(S/\mathcal{J}_{K_3,G})<\mathrm{min}\{2\eta(G),5\}$.} \label{fig2}
\end{figure}

\begin{example}{\rm
We consider the pair of graphs of $(K_{3},G)$, where $G$ is the block graph given in the Figure \ref{fig2}. Note that $c(G)=\eta(G)=2$ and the number of cut vertices of $G$ belong to each maximal clique is $1$. The maximal clique $K_{3}$ does not contain $3+1$ vertices and thus, $G\not\in \mathcal{B}_{3}$. Using Macaulay2 (\cite{mac2}), we  calculate $\mathrm{reg}(S/\mathcal{J}_{K_{3},G})=3$, where $S=\mathbb{Q}[x_{ij}\mid 1\leq i\leq 3,\, 1\leq j\leq 6]$. In this case, $\mathrm{min}\{(m-1)\eta(G), n-1\}=4> \mathrm{reg}(S/\mathcal{J}_{K_{3},G})$. Hence, the assumption of taking the particular class of block graphs in Theorem \ref{thmBm} to attained the upper bound given in Theorem \ref{thmcomp} is justified.
}
\end{example}

\printbibliography

\end{document}